\documentclass[10pt]{article}
\usepackage{amsmath}
\usepackage{amssymb}
\usepackage{amsthm}
\usepackage{amscd}
\usepackage{xypic}
\usepackage{delarray}
\usepackage{mathrsfs}
\usepackage{xcolor}
\usepackage{hyperref}
\usepackage{tikz}
\usetikzlibrary{positioning}

\headsep=1cm \numberwithin{equation}{section}
\def\C{{\mathbb C}}

\def\P{{\mathbb P}}
\def\Q{{\mathbb Q}}

\def\Z{{\mathbb Z}}

\newtheorem{theorem}{Theorem}[section]
\newtheorem{lemma}[theorem]{Lemma}
\newtheorem{proposition}[theorem]{Proposition}
\newtheorem{corollary}[theorem]{Corollary}

\theoremstyle{definition}
\newtheorem{definition}[theorem]{Definition}
\newtheorem{remark}[theorem]{Remark}

\newtheorem{convention and reminder}[theorem]{Convention and Reminder}
\newtheorem{convention and remark}[theorem]{Convention and Remark}
\newtheorem{definition and remark}[theorem]{Definition and Remark}

\newtheorem{reminders and definition}[theorem]{Reminders and Definition}

\newtheorem{notation and remarks}[theorem]{Notation and Remarks}
\newtheorem{notation and remark}[theorem]{Notation and Remark}
\newtheorem{example}[theorem]{Example}

\newtheorem{problem}[theorem]{Problem}

\begin{document}
\title{Non-Jordaness of the automorphism group of the zero-divisor graph of a matrix ring over number rings}
\author{WonTae Hwang
  \and
  Ei Thu Thu Kyaw$^{\dagger}$\footnote{$\dagger$ Corresponding author}
}
\newcommand{\Addresses}{{
  \bigskip
  \footnotesize

  WonTae Hwang, \textsc{Department of Mathematics and Institute of Pure and Applied Mathematics, Jeonbuk National University, Baekje-daero 567, Deokjin-gu,
    Jeonju-si, Jeollabuk-do, South Korea 54896}\par\nopagebreak
  \textit{E-mail address}: \texttt{hwangwon@jbnu.ac.kr} 

\medskip

Ei Thu Thu Kyaw, \textsc{Department of Mathematics, Yonsei University, 50 Yonsei-Ro,
    Seodaemun-Gu, Seoul, South Korea 03722}\par\nopagebreak
  \textit{E-mail address}: \texttt{eithu.k@yonsei.ac.kr} 

}}

\maketitle
\begin{abstract}
  We provide a construction of the induced subgraphs of the zero-divisor graph of $M_2(R)$ for the ring $R$ of algebraic integers of some number fields that are neither complete nor connected, and study the structure of the induced subgraphs explicitly. As an application, we prove that the automorphism group of the zero-divisor graph of $M_2(R)$ is not a Jordan group.
\end{abstract}

\section{Introduction}

Let $S$ be any ring with identity, $Z(S)$ the set of all zero-divisors (either left or right) of $S$, and let $Z(S)^\times = Z(S)\setminus \{0\}$. 
To such an $S$, we can associate the \emph{zero-divisor graph} of $S$, denoted by $\Gamma (S)$, whose vertices are the elements in $Z(S)^{\times}$, and in which for any two distinct vertices $v_1$ and $v_2$ in $\Gamma(S)$, they are adjacent if and only if $v_1 v_2=0$. The notion of the zero-divisor graph of a commutative ring (with identity) $S$ was first given by I. Beck \cite{7}, where he studied the coloring of the zero-divisor graph of $S$, whose vertex set was defined to be $S$ itself rather than the set $Z(S)^{\times}.$ (See also \cite{8} for a continuation of \cite{7}.) A more refined and relevant definition of the zero divisor graph $\Gamma(S)$ of a commutative ring $S$ (with identity) to that of ours appeared in \cite{3} by D. F. Anderson and P. S. Livingston, where they studied the connection between the ring-theoretic property of $S$ and the graph-theoretic property of its zero-divisor graph $\Gamma(S)$ in some extent. We note that the definition of the zero-divisor graph of a commutative ring $S$ in \cite{3} naturally extends to that of non-commutative rings (see \cite{9}). In this regard, the zero-divisor graph of commutative rings has been studied extensively to describe its graph-theoretic structure (including its automorphism group), or to reveal the close relation between its graph-theoretic properties and the ring-theoretic properties of the ring itself. (For instance, see \cite{12}, \cite{10}, \cite{11}, and \cite{13}.) On the other hand, for the case of $S$ being non-commutative, it seems to the authors that many cases that were dealt with involve matrix rings over finite fields, which, in turn, results in finite ambient zero-divisor graphs. (For example, see \cite{14}, \cite{6}, \cite{1}, and \cite{15}.) With somewhat similar but essentially different point of view, in this paper, we are interested in the zero-divisor graph of the matrix ring $M_2(R)$ for the ring $R$ of algebraic integers in a number field $K$ with $[K : \Q] \leq 2.$ Note that $Z(M_2(R))^{\times}$ is infinite so that the zero-divisor graph $\Gamma(M_2(R))$ is an infinite graph. (For a proof, see Corollary \ref{inf cor} below.) Since it is usually not easy to study infinite graphs, we first describe certain induced subgraphs of $\Gamma(M_2(R))$ that are obtained from a curve lying on the Segre quadric surface in a projective space, called the twisted cubic curve, to derive some properties of the ambient graph $\Gamma(M_2(R))$. In this sense, the authors hope that this paper might be regarded as an attempt to use classical algebraic geometry to study the induced subgraphs of certain infinite zero-divisor graphs. 
In this aspect, one of our main results is the following description on the induced subgraphs of $\Gamma(M_2(R))$ that are obtained from the twisted cubic curve.
\begin{theorem}\label{intro main 1}
Let $R$ be the ring of integers of a number field $K=\Q(\sqrt{-d})$ for either $d=0$ or $d>0$ a square-free integer. Consider the following subset
$$S_{\textrm{tc}}(R) = \left\{ \lambda \begin{bmatrix} 1 & t \\ t^2 & t^3 \end{bmatrix} ~|~ \lambda \in R \setminus \{0\}, t \in R \right\} \cup \left\{\begin{bmatrix} 0 & 0 \\ 0 & \lambda \end{bmatrix} ~|~\lambda \in R \setminus \{0\} \right\}$$
 of $M_2(R).$ Then the induced subgraph $\Gamma_{S_{tc}(R)}(M_2(R))$ of $\Gamma(M_2(R))$ whose vertices lie in $S_{tc}(R)$ has an explicit decomposition as a union of its connected components, depending on whether $d \ne 1,3,$ or $d=1,$ or $d=3.$ 
\end{theorem}
For a more detailed description on the induced subgraph $\Gamma_{S_{tc}(R)}(M_2(R)),$ see Theorem \ref{main thm1} below. The proof of Theorem \ref{intro main 1} is achieved by a direct computation, together with a fundamental description on the units of imaginary quadratic number rings.
\vskip 0.1in

Now, given a graph $\Gamma$, we are also interested in the structure of the automorphism group $\textrm{Aut}(\Gamma)$ of $\Gamma$. In this regard, the following is another main result of this paper, which essentially tells us that the automorphism group $\textrm{Aut}(\Gamma(M_2(R)))$ is rather ``large" in terms of the Jordan property.
\begin{theorem}\label{intro main 2}
Let $R$ be the ring of integers of a number field $K=\Q(\sqrt{-d})$ for either $d=0$ or $d>0$ a square-free integer, and let $G= \textrm{Aut}(\Gamma(M_2 (R))).$ Then $G$ is not a Jordan group.
\end{theorem}
For the definition of Jordan groups, see Definition \ref{Jordan}, and for a proof of Theorem \ref{intro main 2}, see Theorem \ref{main thm2} below. The proof of Theorem \ref{intro main 2} makes use of the observation that the automorphism group $\textrm{Aut}(\Gamma(M_2(R)))$ as above contains the symmetric group $S_n$, whence the alternating group $A_n$ for every $n \geq 1,$ and the well-known fact that the group $A_n$ is simple for any $n \geq 5.$ Theorem \ref{intro main 2} is quite interesting because there are infinite groups that are indeed Jordan groups, which indicates that the proof of Theorem \ref{intro main 2} is not trivial at all. 
\vskip 0.1in

This paper is organized as follows: In Section \ref{preliminary}, we recall some of the facts in the general theory of the zero-divisor graph of a ring (with emphasis on the ring of $(2 \times 2)$ matrices), and the theory of the twisted cubic curve in a projective space, together with some fundamental results on the units of imaginary quadratic number rings. In Section \ref{main}, we obtain several results on the description of certain induced subgraphs of the zero-divisor graph of $M_2(R)$ with $R$ being the ring of integers of a number field, which includes the computations of the girth, chromatic number, clique number, and vertex connectivity of parts of those induced subgraphs (see Theorems \ref{main thm1}, \ref{chromatic thm}, \ref{clique thm}, \ref{connectivity thm} below) using the facts that were introduced in Section \ref{preliminary}. Finally, in Section \ref{aut}, we introduce one group theoretic property (namely, the Jordan property) of arbitrary groups, which somewhat reveals the subgroup structure of the ambient groups, and at the end, we show that the automorphism group of the zero-divisor graph of $M_2(R)$ with $R$ being the ring of integers of certain number fields does not satisfy the Jordan property (see Theorem \ref{main thm2} below).
\vskip 0.1in

Throughout the paper, let $R$ be the ring of integers of $K=\Q(\sqrt{-d})$ either for $d=0$ or $d>0$ a square-free integer, unless otherwise specified. Note that we have $\mathbb{Z} \subseteq R \subseteq \mathbb{C}$ for any such $d$. For an algebraic variety $X$ over the field of complex numbers $\C$, we denote the set of $\C$-points by $X(\C).$ For any set $H,$ $|H|$ denotes the number of elements of $H$, and for an integer $n \geq 1,$ $S_n$ (resp.\ $A_n$) is the symmetric group (resp.\ alternating group) on $n$ letters, and $K_n$ (resp.\ $K_{n,n}$) is a complete graph with $n$ vertices (resp.\ a complete bipartite graph with $2n$ vertices). Also, we denote the zero-divisor graph of a ring $S$ with identity by $\Gamma(S),$ and for a subset $T $ of the set of vertices in $\Gamma(S)$, $\Gamma_T(S)$ denotes the induced subgraph of $\Gamma(S)$ whose vertices lie on $T.$ Finally, for a graph $\Gamma$, we adopt the following notations in the sequel:
\vskip 0.05in
$\bullet$ If $\Gamma$ is a directed graph, then $\Gamma^{\textrm{un}}$ is the underlying undirected graph of $\Gamma.$
\vskip 0.05in
$\bullet$ $V(\Gamma)$ is the set of vertices in $\Gamma$.
\vskip 0.05in
$\bullet$ $E(\Gamma)$ is the set of edges in $\Gamma$.
\vskip 0.05in
$\bullet$ $g(\Gamma)$ is the girth of the graph $\Gamma$. 
\vskip 0.05in
$\bullet$ $\chi(\Gamma)$ is the chromatic number of the graph $\Gamma$. 
\vskip 0.05in
$\bullet$ $\omega(\Gamma)$ is the clique number of the graph $\Gamma$. 
\vskip 0.05in
$\bullet$ $\alpha(\Gamma)$ is the independence number of the graph $\Gamma$. 
\vskip 0.05in
$\bullet$ $\kappa(\Gamma)$ is the vertex connectivity of the graph $\Gamma$.
\vskip 0.05in
$\bullet$ $\textrm{Aut}(\Gamma)$ is the automorphism group of the graph $\Gamma.$ 

\section{Preliminaries}\label{preliminary}
\subsection{Zero-divisor graph of a ring}
In this section, we briefly review the theory of zero-divisor graphs of rings with emphasis on the ring of $(2 \times 2)$ matrices over $R$, where $R$ is the ring of integers of a number field $K=\Q(\sqrt{-d})$ either for $d=0$ or $d>0$ a square-free integer. 
\vskip 0.1in
Let $S$ be any ring with identity, $Z(S)$ be the set of all zero-divisors (either left or right) of $S$, and let $Z(S)^{\times}=Z(S) \setminus \{0\}.$ In this situation, we may introduce a graph that is related to $S$ as follows: let $\Gamma(S)$ be the (directed) graph with the vertex set $V(\Gamma(S))=Z(S)^{\times}$, and with the edge set $E(\Gamma(S))$ given by the rule that for any two distinct vertices $v_1$ and $v_2$ of $\Gamma(S)$, there is an edge from $v_1$ to $v_2$ (denoted by $v_1 \rightarrow v_2$) if and only if $v_1 v_2 =0.$ Also, let $\Gamma^{un}(S)$ denote the underlying undirected graph of $\Gamma(S).$ In particular, we have $V(\Gamma(S))= V(\Gamma^{un}(S))$, and any two distinct vertices $v_1$ and $v_2$ of $\Gamma^{un}(S)$ are adjacent if and only if either $v_1 v_2 = 0$ or $v_2 v_1 = 0.$ Note that if $S$ is commutative, then the graph $\Gamma(S)$ might be regarded as an undirected graph. (In this case, we say that $\Gamma^{un}(S)=\Gamma(S)$.) The following result is on the finiteness of the graph $\Gamma(S)$ for an arbitrary ring $S$, which is essentially a restatement of \cite[Theorem II]{4}.

\begin{lemma}\label{infinite graph lem}
Let $S$ be a ring with identity. If $\Gamma(S)$ is finite with $|V(\Gamma(S))| \geq 1$, then $S$ is finite.
\end{lemma}
From now on, let $R$ be the ring of integers of a number field $K=\Q(\sqrt{-d})$ for $d=0$ or $d>0$ a square-free integer. Note that $\Z \subseteq R.$ The following result that the zero-divisor graph of $M_2(R)$ is infinite might be proved in various ways, and here, we provide one such proof that makes use of the above lemma.
\begin{corollary}\label{inf cor}
The graph $\Gamma(M_2(R))$ is infinite.
\end{corollary}
\begin{proof}
First, it is easy to see that the matrix $\begin{bmatrix} 1 & 0 \\ 0 & 0 \end{bmatrix}$ is a zero-divisor in $M_2(R)$ so that we have $|V(\Gamma(M_2 (R)))| \geq 1.$ Suppose on the contrary that $|V(\Gamma(M_2 (R)))|$ is finite. Then by Lemma \ref{infinite graph lem}, it follows that $M_2(R)$ is also finite, which is absurd. Thus we can conclude that the graph $\Gamma(M_2(R))$ is infinite. 
\end{proof}
\begin{remark}
There is another way to see that the graph $\Gamma(M_2(R))$ is an infinite graph using algebraic geometry. Since $K$ is a field, we first recall that a matrix $A \in M_2(K)$ is a zero-divisor if and only if $\det(A)=0.$ Hence we obtain the following identification of sets
$$Z(M_2(K))=\{A \in M_2(K)~|~\textrm{rank}(A) \leq 1 \}.$$
It follows from this observation that $Z(M_2(K))$ is (the set of $K$-rational points of) a determinantal variety $Y$ of dimension $3$ in $M_2 (K) \cong K^4$ so that $Z(M_2(K))$ is an infinite set. Now, we claim that $Y(R)=Z(M_2(R))$. Indeed, it is clear that $Z(M_2(R)) \subseteq Y(R).$ For the reverse inclusion, since $O = \begin{bmatrix} 0 & 0 \\ 0 & 0 \end{bmatrix} \in Z(M_2(R)),$ let $A= \begin{bmatrix} a & b \\ c & d \end{bmatrix} \ne O \in M_2(R)$ with $\textrm{rank}(A) \leq 1.$ Then by regarding $A$ as an element in $M_2(K),$ since $\textrm{rank}(A) \leq 1,$ there is a $\lambda \in K$ such that $\begin{bmatrix} b \\ d \end{bmatrix} = \lambda \cdot \begin{bmatrix} a \\ c \end{bmatrix}.$ Since $K$ is the field of fractions of $R$, we may write $\lambda = \frac{\alpha}{\beta}$ with $\alpha, \beta \in R ~\textrm{and}~ \beta \ne 0.$ Then note that we have
$$A \cdot \begin{bmatrix} -\alpha & -\alpha \\ \beta & \beta \end{bmatrix} =\begin{bmatrix} a & b \\ c & d \end{bmatrix} \cdot \begin{bmatrix} -\alpha & -\alpha \\ \beta & \beta \end{bmatrix}= \begin{bmatrix} a & \lambda a \\ c & \lambda c \end{bmatrix} \cdot \begin{bmatrix} -\alpha & -\alpha \\ \beta & \beta \end{bmatrix}=O,$$
and hence, $A \in Z(M_2(R))$, which, in turn, implies the desired inclusion $Y(R) \subseteq Z(M_2(R)).$ Then since $Y(\Z) \subseteq Y(R)$, and $Y(\Z)$ is an infinite set, we conclude that $Z(M_2(R))$ is also infinite.

\end{remark}
Since it is rather difficult to deal with the infinite graph $\Gamma(M_2(R))$ as a whole, we consider certain induced subgraphs of $\Gamma(M_2(R))$ coming from some algebraic geometric objects to study some of its properties in later sections. More precisely, we construct the aforementioned induced subgraphs of $\Gamma(M_2(R))$ using a classical curve in the projective $3$-space. (For a detailed description on the curve, refer to Section \ref{curve Segre} below.)

\subsection{Twisted cubic curve on the Segre quadric surface in $\mathbb{P}^3$}\label{curve Segre}
In this section, we summarize some essential ingredients from algebraic geometry following \cite{2}, which will be used for our main result of this paper. We begin with recalling one fundamental example of projective varieties, namely, the projective $n$-space for an integer $n \geq 1$. We work with the field of complex numbers $\C$.
\begin{definition}\label{pn def}
For an integer $n \geq 1,$ a \emph{projective $n$-space over $\C$}, denoted by $\mathbb{P}^n$ (or $\mathbb{P}_{\C}^n$), is defined to be the set of equivalence classes of $(n+1)$-tuples $(a_0, \ldots,a_n)$ of elements of $\C$, not all zero, under the equivalence relation given by 
$$(a_0,\ldots,a_n) \sim (\lambda a_0,\ldots,\lambda a_n)$$
for all $\lambda \in \C^{\times} = \C \setminus \{0\}.$ Elements of $\P^n$ are denoted by $[a_0, a_1, \cdots, a_n]$ for $a_0, a_1, \cdots, a_n \in \C.$
\end{definition} 
In particular, we have $[a_0, a_1, \ldots, a_n] = [b_0, b_1, \ldots, b_n]$ in $\P^n$ if there is a $\lambda \in  \C^{\times}$ such that $b_j= \lambda a_j$ for all $0 \leq j \leq n.$ Also, it is worth noting that $[0, 0, \ldots, 0]$ is \emph{not} a point in $\P^n.$
\vskip 0.1in
Now, we consider the following map of projective varieties
$$ \sigma \colon \mathbb{P}^1 \times \mathbb{P}^1 \rightarrow \mathbb{P}^3$$
given by $\sigma ([X_0, X_1], [Y_0, Y_1]) = [X_0 Y_0, X_0 Y_1, X_1 Y_0, X_1 Y_1 ],$ where $[X_0, X_1]$ and $[Y_0, Y_1]$ are the homogeneous coordinates for the respective $\P^1.$ The image of $\sigma$ in $\P^3$ is of our special interest.
\begin{definition}\label{Segre def}
The image of $\mathbb{P}^1 \times \mathbb{P}^1$ via $\sigma$ is called the \emph{Segre surface in $\mathbb{P}^3$}, and it is denoted by $\Sigma_{1,1}.$ 
\end{definition}
In other words, $\Sigma_{1,1}=\sigma(\mathbb{P}^1 \times \mathbb{P}^1) \subseteq \mathbb{P}^3$.
\begin{remark}\label{Segre def eqn}
If $[Z_0, Z_1, Z_2, Z_3]$ is the homogeneous coordinates for $\mathbb{P}^3,$ then it is rather straightforward to see that the Segre surface $\Sigma_{1,1}$ is the zero locus of the single quadratic polynomial $Z_0 Z_3 - Z_1 Z_2$, and hence, $\Sigma_{1,1}$ is also called the \emph{Segre quadric surface in $\mathbb{P}^3$}. Another interesting observation is that the quadratic polynomial $Z_0 Z_3 - Z_1 Z_2$ can be regarded as the determinant of the $2 \times 2$ matrix $\begin{bmatrix} Z_0 & Z_1 \\ Z_2 & Z_3 \end{bmatrix}$.
\end{remark}
Now, we study one classical and basic curve on the Segre quadric surface $\Sigma_{1,1}$, which plays an important role in the sequel.
\begin{definition}\label{twisted cubic curve def}
Let $\nu \colon \mathbb{P}^1 \rightarrow \mathbb{P}^3$ be the map given by $\nu([X_0, X_1])=[X_0^3, X_0^2 X_1, X_0 X_1^2, X_1^3]$, where $[X_0, X_1]$ is the homogeneous coordinates for $\mathbb{P}^1.$ Then the image $C_{\textrm{tc}}$ of $\mathbb{P}^1$ via $\nu$ is called the \emph{twisted cubic curve}.
\end{definition}
In other words, $C_{\textrm{tc}}=\nu(\P^1) \subseteq \P^3.$ The following fact is about a concrete description of the curve $C_{\textrm{tc}}.$
\begin{lemma}\label{tc lem}
We have
$$C_{\textrm{tc}}(\mathbb{C}) = \{[1,t,t^2, t^3]~|~t \in \mathbb{C}\} \cup \{[0,0,0,1]\}.$$
\end{lemma}
\begin{proof}
Let $P=[z_0, z_1, z_2, z_3] \in C_{\textrm{tc}}(\C).$ Then by definition, there exists a point $[x_0, x_1] \in \P^1(\C)$ such that
$$[z_0, z_1, z_2, z_3]=[x_0^3, x_0^2 x_1, x_0 x_1^2, x_1^3].$$
Now, we consider the following two cases. 
\vskip 0.1in
(i) If $x_0 = 0,$ then we have $x_1 \ne 0$ (because $[x_0, x_1] \in \P^1$) and
$$[z_0, z_1, z_2, z_3]=[0, 0, 0, x_1^3] = [0, 0, 0, 1].$$

(ii) If $x_0 \ne 0,$ then we have
$$[z_0, z_1, z_2, z_3]=[x_0^3, x_0^2 x_1, x_0 x_1^2, x_1^3]=[1, x_1/x_0, (x_1/x_0)^2, (x_1/x_0)^3]$$
(by dividing each homogeneous coordinate by $x_0^3$). By letting $t = x_1/x_0,$ we see that
$$[z_0, z_1, z_2, z_3]=[1,t,t^2,t^3]$$
for some $t \in \C.$ Hence, from (i) and (ii), we can conclude that 
$$P \in \{[1,t,t^2, t^3]~|~t \in \mathbb{C}\} \cup \{[0,0,0,1]\},$$
which shows that $C_{\textrm{tc}}(\C) \subseteq \{[1,t,t^2, t^3]~|~t \in \mathbb{C}\} \cup \{[0,0,0,1]\}.$
\vskip 0.1in
The opposite inclusion also follows similarly from the observation that
$$[1,t,t^2,t^3]=\nu([1,t])~~\textrm{and}~~[0,0,0,1]=\nu([0,1]).$$
This completes the proof.
\end{proof}

A useful observation is that the twisted cubic curve lies on the Segre quadric surface as indicated above.
\begin{lemma}\label{tc on Segre lem}
The curve $C_{\textrm{tc}}$ is contained in the Segre quadric surface $\Sigma_{1,1}.$
\end{lemma}
\begin{proof}
In view of Remark \ref{Segre def eqn}, it suffices to check that if $P=[X_0^3, X_0^2 X_1, X_0 X_1^2, X_1^3]$ for some $[X_0, X_1]\in \P^1,$ then it satisfies the defining polynomial equation of $\Sigma_{1,1},$ which is straightforward because we have $(X_0^3 X_1^3) - (X_0^2 X_1 \cdot X_0 X_1^2) =0.$
\end{proof}


\subsection{Units in a quadratic number ring}
In this subsection, we recall the basics of the group of units in a quadratic number ring. 
\vskip 0.1in
Throughout this section, let $K=\Q(\sqrt{-d})$ (for some square-free integer $d > 0$) be an imaginary quadratic number field, and let $R=\mathcal{O}_K$ be the ring of integers of $K.$ Obviously, $R$ is an integral domain and the following fact about $R$ is well known.
\begin{proposition}
We have
$$R = \begin{cases} \Z[\sqrt{-d}] & \textit{if~$d \equiv 1, 2 \pmod{4}$} \\ \Z \left[\frac{-1+\sqrt{-d}}{2} \right] & \textit{if~$d \equiv 3 \pmod{4}$} \end{cases}.$$ 
\end{proposition}

Now, we recall that a \emph{unit} in $R$ is an element $x \in R$ such that there is an element $y \in R$ with $xy = yx =1.$ If we try to find the set $R^{\times}$ of all the units in $R,$ then we obtain another well-known result.
\begin{proposition}\label{unit prop}
We have
$$R^{\times} = \begin{cases} \{\pm 1\} & \textit{if~$d \ne 1, 3$} \\ \{\pm 1, \pm i  \}& \textit{if~$d =1$} \\ \{\pm 1, \pm \omega, \pm \omega^2\}  & \textit{if~$d =3$} \end{cases}$$
where $i=\sqrt{-1}$ and $\omega$ is a primitive third root of unity. 
\end{proposition}

In particular, as an immediate consequence of Proposition \ref{unit prop}, we get the following observation, which will be used later in the paper.
\begin{corollary}\label{main sol}
Let $x, y \in R$ such that $x^2 y = -1 $ (resp.\ $xy^2 =-1$). Then the pair $(x,y)$ belongs to the following list: 
\vskip 0.1in
(a) If $d \ne 1,3$, then $(x,y) \in \{(1,-1), (-1,-1)\}$ (resp.\ $(x,y) \in \{(-1, 1), (-1, -1)\}).$
\vskip 0.1in
(b) If $d = 1$, then $(x,y) \in \{(1, -1), (-1, -1), (i, 1), (-i, 1)\}$ (resp.\ $(x,y) \in \{(-1, 1), (-1, -1), (1, i), (1, -i)\}$) where $i=\sqrt{-1}$.
\vskip 0.1in
(c) If $d=3,$ then $(x,y) \in \{(1, -1), (-1, -1), (\omega, -\omega), (-\omega, -\omega), (\omega^2, -\omega^2), (-\omega^2, -\omega^2) \}$ (resp.\ $(x,y) \in \{(-1, 1), (-1, -1), (-\omega, \omega), (-\omega, -\omega), (-\omega^2, \omega^2), (-\omega^2, -\omega^2) \}$) where $\omega$ is a primitive third root of unity.
 
\end{corollary}

\section{An induced subgraph of the zero-divisor graph of $M_2 (R)$ that comes from the twisted cubic curve}\label{main}
Let $K=\Q(\sqrt{-d})$ either for $d=0$ or a square-free integer $d>0$, and let $R=\mathcal{O}_K$, the ring of integers of $K$. In this section, we describe an induced subgraph of the zero-divisor graph of the $2 \times 2$ matrix ring over $R.$ To this aim, we first consider the following map
$$\varphi \colon  M_2(\mathbb{C}) \cong \mathbb{C}^4 \hookrightarrow \mathbb{P}^4 (\C) \rightarrow \mathbb{P}^3 (\C),~~~ \begin{bmatrix} a & b \\ c & d \end{bmatrix} \mapsto [a, b, c, d, 1] \mapsto [a, b, c, d]. $$
Recall that the twisted cubic curve $C_{tc}$ lies on the Segre quadric surface $\Sigma_{1,1} \subseteq \mathbb{P}^3$, and that we have $C_{\textrm{tc}}(\mathbb{C})=\{[1,t,t^2, t^3]~|~t \in \mathbb{C}\} \cup \{[0,0,0,1]\}$ (see Lemma \ref{tc lem}).



Now, let $$S_{\textrm{tc}} (\C)= \varphi^{-1}(C_{\textrm{tc}}(\mathbb{C})) = \left\{ \lambda \begin{bmatrix} 1 & t \\ t^2 & t^3 \end{bmatrix} ~|~ \lambda \in \mathbb{C}^{\times}, t \in \mathbb{C} \right\} \cup \left\{\begin{bmatrix} 0 & 0 \\ 0 & \lambda \end{bmatrix} ~|~\lambda \in \mathbb{C}^{\times} \right\},$$ 
and similarly, we let
$$S_{\textrm{tc}}(R) = \left\{ \lambda \begin{bmatrix} 1 & t \\ t^2 & t^3 \end{bmatrix} ~|~ \lambda \in R \setminus \{0\}, t \in R \right\} \cup \left\{\begin{bmatrix} 0 & 0 \\ 0 & \lambda \end{bmatrix} ~|~\lambda \in R \setminus \{0\} \right\}.$$
Clearly, we have $S_{\textrm{tc}}(R) \subseteq S_{\textrm{tc}}(\C) \cap M_2 (R).$
For ease of notation, we also let
\begin{equation}\label{eqn2}
S_{\textrm{tc},1}(R)= \left\{ \lambda \begin{bmatrix} 1 & t \\ t^2 & t^3 \end{bmatrix} ~|~ \lambda \in R \setminus \{0\}, t \in R \right\}, ~~S_{\textrm{tc}, 2}(R) = \left\{\begin{bmatrix} 0 & 0 \\ 0 & \lambda \end{bmatrix} ~|~\lambda \in R \setminus \{0\} \right\}, ~~\textrm{and} 
\end{equation}
\begin{equation}\label{eqn1}
S_{\textrm{tc}, 1, t}(R) = \left \{\lambda \begin{bmatrix} 1 & t \\ t^2 & t^3 \end{bmatrix} ~|~ \lambda \in R \setminus \{0\} \right\}~\textrm{for a fixed $t \in R$}.
\end{equation} 
With these notations, it is straightforward to see that
$$S_{\textrm{tc}}(R)=S_{\textrm{tc},1}(R) \cup S_{\textrm{tc}, 2}(R)~~\textrm{and}~~S_{\textrm{tc},1}(R)=\bigcup_{t \in R} S_{\textrm{tc},1,t}(R).$$

The following fact is standard, and we give a rather computational proof for later use. 
\begin{lemma}\label{Stc subset lem}
We have $S_{\textrm{tc}}(R) \subseteq Z(M_2(R)).$
\end{lemma}
\begin{proof}
Let $A \in S_{\textrm{tc}}(R).$ We have two cases to consider.
\vskip 0.1in
(i) Let $A=\lambda \begin{bmatrix} 1 & t \\ t^2 & t^3 \end{bmatrix}$ for some $\lambda \in R \setminus \{0\}$ and $t \in R$. We find suitable matrices $B=\begin{bmatrix} a & b \\ c & d \end{bmatrix} \in M_2(R)$ and $C=\begin{bmatrix} \alpha & \beta \\ \gamma & \delta \end{bmatrix} \in M_2(R)$ such that $AB=CA=O.$ Indeed, a direct computation shows that we may take $B=\begin{bmatrix} -ct & -dt \\ c & d \end{bmatrix}$ where $c, d \in R$ with one of them being nonzero, and $C=\begin{bmatrix} -\beta t^2 & \beta \\ - \delta t^2 & \delta \end{bmatrix}$ where $\beta, \delta \in R$ with one of them being nonzero.
\vskip 0.1in
Similarly,
\vskip 0.1in
(ii) Let $A=\begin{bmatrix} 0 & 0 \\ 0 & \lambda \end{bmatrix}$ for some $\lambda \in R \setminus \{0\}.$ Then we may take $B=\begin{bmatrix} a & b \\ 0 & 0 \end{bmatrix}$ for some $a, b \in R$ with one of them being nonzero, and $C=\begin{bmatrix} \alpha & 0 \\ \gamma & 0 \end{bmatrix}$ where $\alpha, \gamma \in R$ with one of them being nonzero.
\vskip 0.1in
Hence it follows from (i) and (ii) that $S_{\textrm{tc}}(R) \subseteq Z(M_2(R))$, as desired.
\end{proof}

In view of Lemma \ref{Stc subset lem}, we can consider the induced subgraph $\Gamma_{S_{\textrm{tc}}(R)}(M_2 (R))$ of $\Gamma(M_2 (R)).$ Then the following is one of our main results.


\begin{theorem}\label{main thm1}
The graph $\Gamma_{S_{\textrm{tc}}(R)}(M_2 (R))$ consists of the following (depending on $d$):
\vskip 0.1in
(Case 1) If $d \ne 1,3$, then
\vskip 0.1in
(1) Component $\Gamma_1$ whose vertices are in $S_{\textrm{tc}, 1,0}(R)$ and $S_{\textrm{tc},2}(R)$, which is an (undirected) infinite complete bipartite graph with $g(\Gamma_1)=4$;
\vskip 0.1in
(2) Component $\Gamma_2$ whose vertices are in $S_{\textrm{tc}, 1,-1}(R)$ and $S_{\textrm{tc},1, 1}(R)$, which is a (directed) infinite graph with $g(\Gamma^{un}_2)=3$ (where $\Gamma^{un}_2$ is the underlying undirected graph of $\Gamma_2$);
\vskip 0.1in
(3) Union of all the remaining $K_1$'s whose vertices are in $S_{\textrm{tc}, 1,t}(R)$ with $t \ne 0, \pm 1 \in R$.
\vskip 0.2in
(Case 2) If $d = 1$, then (after recalling that $i = \sqrt{-1}$)
\vskip 0.1in
(1) Component $\Gamma_1$ whose vertices are in $S_{\textrm{tc}, 1,0}(R)$ and $S_{\textrm{tc},2}(R)$, which is an (undirected) infinite complete bipartite graph with $g(\Gamma_1)=4$;
\vskip 0.1in
(2) Component $\Gamma_2$ whose vertices are in $S_{\textrm{tc}, 1,-1}(R)$, $S_{\textrm{tc},1, 1}(R)$, $S_{\textrm{tc}, 1, i}(R)$, and $S_{\textrm{tc}, 1, -i}(R)$ which is a (directed) infinite graph with $g(\Gamma^{un}_2)=3$ (where $\Gamma^{un}_2$ is the underlying undirected graph of $\Gamma_2$);
\vskip 0.1in
(3) Union of all the remaining $K_1$'s whose vertices are in $S_{\textrm{tc}, 1,t}(R)$ with $t \ne 0, \pm 1, \pm i \in R$.

\vskip 0.2in
(Case 3) If $d = 3$, then (after recalling that $\omega$ is a primitive third of unity)
\vskip 0.1in
(1) Component $\Gamma_1$ whose vertices are in $S_{\textrm{tc}, 1,0}(R)$ and $S_{\textrm{tc},2}(R)$, which is an (undirected) infinite complete bipartite graph with $g(\Gamma_1)=4$;
\vskip 0.1in
(2) Component $\Gamma_2$ whose vertices are in $S_{\textrm{tc}, 1,-1}(R)$, $S_{\textrm{tc},1, 1}(R)$ which is a (directed) infinite graph with $g(\Gamma^{un}_2)=3$ (where $\Gamma^{un}_2$ is the underlying undirected graph of $\Gamma_2$);
\vskip 0.1in
(3) Component $\Gamma_3$ whose vertices are in $S_{\textrm{tc}, 1,\omega}(R)$, $S_{\textrm{tc},1, -\omega}(R)$ which is a (directed) infinite graph with $g(\Gamma^{un}_3)=3$ (where $\Gamma^{un}_3$ is the underlying undirected graph of $\Gamma_3$);
\vskip 0.1in
(4) Component $\Gamma_4$ whose vertices are in $S_{\textrm{tc}, 1,\omega^2}(R)$, $S_{\textrm{tc},1, -\omega^2}(R)$ which is a (directed) infinite graph with $g(\Gamma^{un}_4)=3$ (where $\Gamma^{un}_4$ is the underlying undirected graph of $\Gamma_4$);
\vskip 0.1in
(5) Union of all the remaining $K_1$'s whose vertices are in $S_{\textrm{tc}, 1,t}(R)$ with $t \ne 0, \pm 1, \pm \omega, \pm \omega^2 \in R$.

\end{theorem}
\begin{proof}
By the definition of the induced subgraph $\Gamma_{S_{\textrm{tc}}(R)}(M_2 (R)),$ it suffices to see when those matrices in the proof of Lemma \ref{Stc subset lem} lie on the subset $S_{\textrm{tc}}(R).$ We consider the following four cases (following the notations in (\ref{eqn2}) above).
\vskip 0.1in
(i) If $B=\begin{bmatrix} -ct & -dt \\ c & d \end{bmatrix} \in S_{\textrm{tc},1}(R)$ (so that $B$ corresponds to a matrix $A=\lambda \begin{bmatrix} 1 & t \\ t^2 & t^3 \end{bmatrix})$, then we may write
$$B=\begin{bmatrix} -ct & -dt \\ c & d \end{bmatrix} = \begin{bmatrix} \mu & \mu s \\ \mu s^2 & \mu s^3 \end{bmatrix}$$
for some $\mu \in R \setminus \{0\}$ and $s \in R.$ By equating, we get an equation $c (1+ts^2)=0$. If $c=0,$ then $\mu=0$, which is absurd, and hence, it follows that $c \ne 0$ and $1+t s^2 =0.$ By Corollary \ref{main sol}, we see that if $d \ne 1,3$ (resp.\ if $d=1$, resp. if $d=3$), then $t=-1, s=\pm 1$ (resp.\ $t=-1, s=\pm 1$ or $t=1, s=\pm i$, resp.\ $t=-1, s=\pm 1$ or $t=-\omega, s=\pm \omega$ or $t=-\omega^2, s= \pm \omega^2$).
\vskip 0.1in 

Now, if $C=\begin{bmatrix} -\beta t^2 & \beta \\ -\delta t^2 & \delta \end{bmatrix} \in S_{\textrm{tc},1}(R)$ (so that $C$ corresponds to a matrix $A=\lambda \begin{bmatrix} 1 & t \\ t^2 & t^3 \end{bmatrix})$, then we may write
$$C=\begin{bmatrix} -\beta t^2 & \beta \\ -\delta t^2 & \delta \end{bmatrix} = \begin{bmatrix} \mu & \mu s \\ \mu s^2 & \mu s^3 \end{bmatrix}$$
for some $\mu \in R \setminus \{0\}$ and $s \in R.$ By equating, we get an equation $\beta (1+st^2)=0$. If $\beta=0,$ then $s=0=\delta$ (because $\mu \ne 0$) so that $C=O$, which is absurd. If $\beta \ne 0,$ then we have $1+st^2 =0.$ By Corollary \ref{main sol}, we see that if $d \ne 1,3$ (resp.\ if $d=1$, resp. if $d=3$), then $t=\pm 1, s=-1$ (resp.\ $t=\pm1, s=-1$ or $t=\pm i, s=1$, resp.\ $t=\pm 1, s=-1$ or $t= \pm \omega, s=- \omega$ or $t=\pm \omega^2, s= -\omega^2$).
\vskip 0.1in
(ii) If $B=\begin{bmatrix} -ct & -dt \\ c & d \end{bmatrix} \in S_{\textrm{tc},2}(R)$ (so that $B$ corresponds to a matrix $A=\lambda \begin{bmatrix} 1 & t \\ t^2 & t^3 \end{bmatrix})$, then we may write
$$B=\begin{bmatrix} -ct & -dt \\ c & d \end{bmatrix} = \begin{bmatrix} 0 & 0 \\ 0 & \mu  \end{bmatrix}$$
for some $\mu \in R \setminus \{0\}$. By equating, we get an equation $d = \mu \ne 0$, and hence, we have $t=0$ (and the corresponding $A=\begin{bmatrix} \lambda & 0 \\ 0 & 0 \end{bmatrix}$). 
\vskip 0.1in 
Now, if $C=\begin{bmatrix} -\beta t^2 & \beta \\ -\delta t^2 & \delta \end{bmatrix} \in S_{\textrm{tc},2}(R)$ (so that $C$ corresponds to a matrix $A=\lambda \begin{bmatrix} 1 & t \\ t^2 & t^3 \end{bmatrix})$, then we may write
$$C=\begin{bmatrix} -\beta t^2 & \beta \\ -\delta t^2 & \delta \end{bmatrix} = \begin{bmatrix} 0 & 0 \\ 0 & \mu \end{bmatrix}$$
for some $\mu \in R \setminus \{0\}$. By equating, we get an equation $\delta = \mu \ne 0$, and hence, we have $t^2=0$, whence $t=0$ (and the corresponding $A=\begin{bmatrix} \lambda & 0 \\ 0 & 0 \end{bmatrix}$).

\vskip 0.1in
(iii) If $B=\begin{bmatrix} a & b \\ 0 & 0 \end{bmatrix} \in S_{\textrm{tc},1}(R)$ (so that $B$ corresponds to a matrix $A=\begin{bmatrix} 0 & 0 \\ 0 & \lambda \end{bmatrix})$, then we may write
$$B=\begin{bmatrix} a & b \\ 0 & 0 \end{bmatrix} = \begin{bmatrix} \mu & \mu s \\ \mu s^2 & \mu s^3 \end{bmatrix}$$
for some $\mu \in R \setminus \{0\}$ and $s \in R.$ By equating, we get $a = \mu \ne 0$ and $s=0$ (because $\mu \ne 0$), whence $b=0$. 
\vskip 0.1in 
Now, if $C=\begin{bmatrix} \alpha & 0 \\ \gamma & 0 \end{bmatrix} \in S_{\textrm{tc},1}(R)$ (so that $C$ corresponds to a matrix $A= \begin{bmatrix} 0 & 0 \\ 0 & \lambda \end{bmatrix})$, then we may write
$$C=\begin{bmatrix} \alpha & 0 \\ \gamma & 0  \end{bmatrix} = \begin{bmatrix} \mu & \mu s \\ \mu s^2 & \mu s^3 \end{bmatrix}$$
for some $\mu \in R \setminus \{0\}$ and $s \in R.$ By equating, we get $\alpha = \mu \ne 0$ and $s=0$ (because $\mu \ne 0$), whence $\gamma=0$.
\vskip 0.1in 
(iv) If $B=\begin{bmatrix} a & b \\ 0 & 0 \end{bmatrix} \in S_{\textrm{tc},2}(R)$ (so that $B$ corresponds to a matrix $A=\begin{bmatrix} 0 & 0 \\ 0 & \lambda \end{bmatrix})$, then we may write
$$B=\begin{bmatrix} a & b \\ 0 & 0 \end{bmatrix} = \begin{bmatrix} 0 & 0 \\ 0 & \mu  \end{bmatrix}$$
for some $\mu \in R \setminus \{0\}$. By equating, we get $\mu =0$, which is absurd. 
\vskip 0.1in 
Now, if $C=\begin{bmatrix} \alpha & 0 \\ \gamma & 0 \end{bmatrix} \in S_{\textrm{tc},2}(R)$ (so that $C$ corresponds to a matrix $A= \begin{bmatrix} 0 & 0 \\ 0 & \lambda \end{bmatrix})$, then we may write
$$C=\begin{bmatrix} \alpha & 0 \\ \gamma & 0  \end{bmatrix} = \begin{bmatrix} 0 & 0 \\ 0 & \mu  \end{bmatrix}$$
for some $\mu \in R \setminus \{0\}$. By equating, we get $\mu =0$, which is again absurd.

\vskip 0.1in
Consequently, it follows from (i) that we obtain (Case 1)-(2), (Case 2)-(2), and (Case 3)-(2),(3),(4), and from (ii), (iii) that we obtain (Case 1)-(1), (Case 2)-(1), and (Case 3)-(1). Also from all the cases (i)-(iv), we obtain (Case 1)-(3), (Case 2)-(3), and (Case 3)-(5).
\vskip 0.1in
 This completes the proof.
\end{proof}

\begin{corollary}\label{not connected lem}
$\Gamma_{S_{\textrm{tc}}(R)}(M_2 (R))$ (and hence, $\Gamma(M_2 (R))$) is not complete.
\end{corollary}

\begin{remark}\label{ring rmk1}
The graph $\Gamma_{S_{\textrm{tc}}(R)}(M_2(R))$ has four non-trivial (that is, non-$K_1$'s) connected components if and only if $d = 3.$  
\end{remark}

To give one consequence of Corollary \ref{not connected lem}, we begin with introducing some terminologies following \cite{1}. Let $X=X(M_2(R))$ be the set of all nonzero nonunits of the ring $M_2(R),$ and let $G=G(M_2(R))$ be the group of all units in $M_2(R).$ Note that $Z(M_2(R))^{\times} \subseteq X.$ Now, $G$ naturally acts on $X$ by $G \times X \rightarrow X, ~~(g,x) \mapsto gx,$ which is called the \emph{left regular action of $G$ on $X.$} In this situation, we first have the following result, which is essentially the same as \cite[Proposition 2.1]{1}.
\begin{lemma}\label{trans comp lem}
If the left regular action of $G$ on $X$ is transitive, then the zero-divisor graph $\Gamma(M_2(R))$ is complete.
\end{lemma}
\begin{proof}
Let $x, y \in Z(M_2(R))^{\times}.$ Then clearly, $x, y \in X$, and then, since the left regular action of $G$ on $X$ is transitive, there exists a $g \in G$ such that $x=gy.$ It follows that we have $xy = g y^2 = 0$ in view of \cite[Remark 1]{1}. A similar argument shows that $y x =0.$ Hence, we can conclude that $\Gamma(M_2(R))$ is complete.
\end{proof}
From Corollary \ref{not connected lem} and Lemma \ref{trans comp lem}, we obtain the following fact.
\begin{corollary}
The left regular action of $G$ on $X$ is not transitive.
\end{corollary}
\begin{remark}
The same assertion holds for the right regular action of $G$ on $X$ (that is given by $G \times X \rightarrow X$, $(g,x) \mapsto x g^{-1}$).
\end{remark}

We give another aspect in terms of the connectedness of the subgraph in the next remark.
\begin{remark}
By Theorem \ref{main thm1}, the graph $\Gamma_{S_{\textrm{tc}}(R)} (M_2 (R))$ is not connected. On the other hand, in view of \cite[Theorem 3.1]{6}, the graph $\Gamma(M_2(R))$ is connected. Hence the graph $\Gamma_{S_{\textrm{tc}}(R)}(M_2(R))$ provides one nontrivial example where an induced subgraph of a connected graph might not be connected.
\end{remark}


From the infinite graph $\Gamma_{S_{tc}(R)}(M_2(R)),$ we can further consider its finite induced subgraphs by restricting the values of $\lambda.$ To this aim, we first introduce the following subsets of $S_{\textrm{tc}}(R).$
\vskip 0.1in
\begin{definition}\label{level def}
For an integer $N \geq 1$, we let 
$$S_{tc,1,t,N}= \left \{ \lambda \begin{bmatrix} 1 & t \\ t^2 & t^3 \end{bmatrix} ~|~   -N \leq \lambda \leq N, \lambda \in \Z \setminus \{0\}, t \in R \right \}$$ and   
$$S_{tc,2,N}=\left \{ \begin{bmatrix} 0 & 0 \\ 0 & \lambda \end{bmatrix}~|~ -N \leq \lambda \leq N, \lambda \in \Z \setminus \{0\} \right \}.$$
\end{definition}
\begin{remark}
It might be also natural to consider the following sets instead of $S_{\textrm{tc},1,t,N}$ and $S_{\textrm{tc},2,N}$: for an integer $N \geq 1,$ let
$$S^{\prime}_{\textrm{tc},1,t,N} = \left \{ \lambda \begin{bmatrix} 1 & t \\ t^2 & t^3 \end{bmatrix} ~|~   Nr_{K/\Q}(\lambda) \leq N, \lambda \in R \setminus \{0\}, t \in R \right \}$$ and
$$S^{\prime}_{tc,2,N}=\left \{ \begin{bmatrix} 0 & 0 \\ 0 & \lambda \end{bmatrix}~|~ Nr_{K/\Q}(\lambda) \leq N, \lambda \in R \setminus \{0\} \right \}$$
where $Nr_{K/\Q}$ denotes the field norm of $K$ over $\Q.$ It turns out that those two sets are considerably harder to describe precisely. As an example, if $d=1$ so that $K= \Q(i),$ $R=\Z[i]$, and $\lambda = a+ bi \in \Z[i],$ then we have $Nr_{K/\Q}(\lambda)=a^2 +b^2.$ Hence, if $Nr_{K/\Q}(\lambda) \leq N,$ then we need to consider the inequality $a^2 + b^2 \leq N,$ which turns out to be very difficult to solve as $N$ tends to infinity.
\end{remark}
In terms of graphs, we also give the following.
\begin{definition}\label{subgraph def}
For an integer $N \geq 1,$ we define the following induced subgraphs of $\Gamma_{S_{tc}(R)}(M_2(R))$ (depending on $d$):
\vskip 0.1in
(Case 1) If $d \ne 1, 3$, then we let $\Gamma_{1,N}$ be the induced subgraph of $\Gamma_{S_{tc}(R)}(M_2(R))$ whose vertices lie on the set $S_{\textrm{tc}, 1, 0, N} \cup S_{\textrm{tc},2,N},$ and let $\Gamma_{2,N}$ be the induced subgraph of $\Gamma_{S_{tc}(R)}(M_2(R))$ whose vertices lie on the set $\displaystyle \bigcup_{j \in \{\pm 1\}}S_{\textrm{tc}, 1, j, N}.$
\vskip 0.2in
(Case 2) If $d = 1$, then we let $\Gamma_{1,N}$ be the induced subgraph of $\Gamma_{S_{tc}(R)}(M_2(R))$ whose vertices lie on the set $S_{\textrm{tc}, 1, 0, N} \cup S_{\textrm{tc},2,N},$ and let $\Gamma_{2,N}$ be the induced subgraph of $\Gamma_{S_{tc}(R)}(M_2(R))$ whose vertices lie on the set $\displaystyle \bigcup_{j \in \{\pm 1, \pm i\}}S_{\textrm{tc}, 1, j, N},$ where $i=\sqrt{-1}$. 
\vskip 0.2in
(Case 3) If $d = 3$, then we let $\Gamma_{1,N}$ be the induced subgraph of $\Gamma_{S_{tc}(R)}(M_2(R))$ whose vertices lie on the set $S_{\textrm{tc}, 1, 0, N} \cup S_{\textrm{tc},2,N},$ and let $\Gamma_{2,N}$ be the induced subgraph of $\Gamma_{S_{tc}(R)}(M_2(R))$ whose vertices lie on the set $\displaystyle \bigcup_{j \in \{\pm 1, \pm \omega, \pm \omega^2 \}}S_{\textrm{tc}, 1, j, N},$ where $\omega$ is a primitive third root of unity. 

\end{definition}
Definitely, the graphs $\Gamma_{1,N}$ and $\Gamma_{2,N}$ are finite induced subgraphs of $\Gamma_1$ and $\Gamma_2$ in Theorem \ref{main thm1}, respectively, for the case when $d \ne 3$. Also, for $d=3,$ the graphs $\Gamma_{1,N}$ and $\Gamma_{2, N}$ are finite induced subgraphs of $\Gamma_1$ and $\bigcup_{j \in \{2,3,4\}} \Gamma_j,$ respectively. (In particular, $\Gamma_{1, N}$ will be regarded as an undirected graph in the sequel.) To describe $\Gamma_{1,N}$ and $\Gamma_{2,N}$ in more details in Remark \ref{planar ex} below, we first introduce an operation on two graphs.
\begin{definition}
Let $G=(V, E)$ and $G^{\prime}=(V^{\prime}, E^{\prime})$ be two disjoint undirected graphs. Then we define the \emph{join} of $G$ and $G^{\prime}$, denoted by $G * G^{\prime}$, to be the graph that is obtained from the graph $G \cup G^{\prime}$ by joining all the vertices of $G$ to all the vertices of $G^{\prime}.$ 
\end{definition}

Now, the followings are the main descriptions on the graphs $\Gamma_{1,N}$ and $\Gamma_{2,N}$ depending on $d$ and $N.$
\begin{remark}\label{planar ex}
(1) Suppose that $d \ne 1, 3.$ Then we have $\Gamma_{1,1} = K_{2,2}$, and more generally, we note that $\Gamma_{1,n}=K_{2n,2n}$ for all $n \geq 1$. On the other hand, the underlying undirected graph $\Gamma^{un}_{2,1}$ of $\Gamma_{2,1}$ is a simple connected graph with 4 vertices and 5 edges, and the degree sequence $(3, 3, 2, 2),$ (which is isomorphic to the graph $K_4$ minus one edge) and more generally, we note that the underlying undirected graph $\Gamma^{un}_{2,n}$ of $\Gamma_{2,n}$ is a simple connected graph with $4n$ vertices and $6n^2 - n$ edges, and the degree sequence $(\smash[b]{\ \underbrace{4n-1, 4n-1, \cdots, 4n-1}_\text{$2n$ times}},\smash[b]{\ \underbrace{2n, 2n, \cdots, 2n}_\text{$2n$ times}} )$
\vskip 0.25in
\noindent (which is isomorphic to the graph $K_{2n}*(2n K_1)$). In particular, we have that $\Gamma_{1,n} \cup \Gamma^{un}_{2,n}$ is isomorphic to the graph $K_{2n,2n} \cup (K_{2n}*(2n K_1)).$ Similarly,
\vskip 0.1in
(2) Suppose that $d = 1.$ Then we have $\Gamma_{1,1} = K_{2,2}$, and more generally, we note that $\Gamma_{1,n}=K_{2n,2n}$ for all $n \geq 1$. On the other hand, the underlying undirected graph $\Gamma^{un}_{2,1}$ of $\Gamma_{2,1}$ is isomorphic to the graph $2K_1 * (K_2 \cup 2K_1 \cup 2K_1)=2K_1 * (K_2 \cup 4K_1)$, and more generally, we note that the underlying undirected graph $\Gamma^{un}_{2,n}$ of $\Gamma_{2,n}$ is isomorphic to the graph $2n K_1 * (K_{2n}\cup 2n K_1 \cup 2n K_1) = 2n K_1 * (K_{2n} \cup 4n K_1)$. 
\vskip 0.1in
(3) Suppose that $d = 3.$ Then we have $\Gamma_{1,1} = K_{2,2}$, and more generally, we note that $\Gamma_{1,n}=K_{2n,2n}$ for all $n \geq 1$. On the other hand, the underlying undirected graph $\Gamma^{un}_{2,1}$ of $\Gamma_{2,1}$ is isomorphic to the graph $(2K_1 * K_2) \cup (2K_1*K_2) \cup (2K_1 * K_2)$, and more generally, we note that the underlying undirected graph $\Gamma^{un}_{2,n}$ of $\Gamma_{2,n}$ is isomorphic to the graph $(2nK_1 * K_{2n}) \cup (2nK_1*K_{2n}) \cup (2nK_1 * K_{2n})$.
\end{remark}
In light of Remark \ref{planar ex}, we can see that the graphs $\Gamma_{1, n}$ and $\Gamma_{2,n}^{un}$ do not contain the path graph $P_4$ as an induced subgraph, and hence, they are of special type as indicated in the following theorem.
\begin{theorem}
For each $n \geq 1$ and any square-free integer $d,$ $\Gamma_{1,n}$ and $\Gamma_{2,n}^{un}$ are cographs. 
\end{theorem}

The next remark, together with Remark \ref{ring rmk1}, provides one example, where we can relate the property of the ring $R$ and the structure of the graph $\Gamma_{S_{\textrm{tc}}(R)}(M_2 (R))$.
\begin{remark}
Suppose $d \ne 3$ so that $\Gamma_{S_{\textrm{tc}}(R)}(M_2(R))$ has two non-trivial connected components. Then for any integer $N \geq 1,$ the underlying undirected graph of the finite induced subgraph $\Gamma_{2,N}$ of $\Gamma_2$ has $8N$ vertices with the degree sequence 
$$(\smash[b]{\ \underbrace{6N, 6N, \cdots, 6N}_\text{$2N$ times}},\smash[b]{\ \underbrace{4N-1, 4N-1, \cdots, 4N-1}_\text{$2N$ times}}, \smash[b]{\ \underbrace{2N, 2N, \cdots, 2N}_\text{$4N$ times}} )$$
 \vskip 0.1in
\noindent if and only if $d = 1.$ 
\end{remark}
\vskip 0.1in

Note also that Remark \ref{planar ex} gives a partial answer for the following general question.
\begin{problem}
Classify all the finite graphs that are (isomorphic to) an induced subgraph of the zero-divisor graph $\Gamma(M_2(R))$ for some commutative ring $R.$
\end{problem}
Now, using the above Remark \ref{planar ex}, we can further tell several fundamental properties of the finite graphs $\Gamma_{1,n}$ and $\Gamma^{un}_{2,n}$ for various $n \geq 1.$ The first one is about the property of being planar graphs.
\begin{theorem}
Let $n \geq 1$ be an integer. 
\vskip 0.1in
(a) For any $d,$ the finite graph $\Gamma_{1,n}$ is planar if and only if $n =1.$
\vskip 0.1in
(b) For any $d,$ the underlying undirected graph $\Gamma^{un}_{2,n}$ of $\Gamma_{2,n}$ is planar if and only if $n=1.$

\vskip 0.1in
\end{theorem}
\begin{proof}
(a) By Remark \ref{planar ex}, for each $n \geq 1,$ the graph $\Gamma_{1,n}$ is isomorphic to the complete bipartite graph $K_{2n, 2n}.$ Now, $\Gamma_{1,1} = K_{2,2}$ is clearly planar. For $n \geq 2,$ we see that $\Gamma_{1,n}=K_{2n, 2n}$ contains a $K_{3,3}$ as a subgraph, and hence, $\Gamma_{1,n}$ is not planar by Kuratowski's theorem. 
\vskip 0.1in
Similarly,
\vskip 0.1in
(b) We consider the following three cases.
\vskip 0.1in
(i) If $d \ne 1,3$, then by Remark \ref{planar ex}, for each $n \geq 1,$ $\Gamma^{un}_{2,n}$ is isomorphic to $K_{2n}*(2n K_1).$ Now, $\Gamma^{un}_{2,1}$ is a $K_4$ minus one edge, which is clearly planar. For $n \geq 2,$ we see that $\Gamma^{un}_{2,n} = K_{2n}*(2n K_1)$ contains a $K_{3,3}$ as a subgraph, and hence, it is not planar. Similarly,
\vskip 0.1in
(ii) If $d = 1$, then by Remark \ref{planar ex}, for each $n \geq 1,$ $\Gamma^{un}_{2,n}$ is isomorphic to $2nK_1 * (K_{2n} \cup 4n K_1).$ Now, $\Gamma^{un}_{2,1}=2K_1 * (K_2 \cup 4K_1)$ contains neither a $K_{3,3}$ nor a $K_5$, and hence, it is planar. For $n \geq 2,$ we see that $\Gamma^{un}_{2,n} = 2nK_1 * (K_{2n} \cup 4n K_1)$ contains a $K_{3,3}$ as a subgraph, and hence, it is not planar. 
\vskip 0.1in
(iii) If $d = 3$, then by Remark \ref{planar ex}, for each $n \geq 1,$ $\Gamma^{un}_{2,n}$ is isomorphic to $(2nK_1 * K_{2n}) \cup (2n K_1 * K_{2n}) \cup (2n K_1 * K_{2n}).$ Now, $\Gamma^{un}_{2,1}=(2K_1 * K_{2}) \cup (2 K_1 * K_{2}) \cup (2 K_1 * K_{2})$ contains neither a $K_{3,3}$ nor a $K_5$, and hence, it is planar. For $n \geq 2,$ we see that $\Gamma^{un}_{2,n} = (2nK_1 * K_{2n}) \cup (2n K_1 * K_{2n}) \cup (2n K_1 * K_{2n})$ contains a $K_{3,3}$ as a subgraph, and hence, it is not planar.
\vskip 0.1in
This completes the proof. 
\end{proof}
The second one is about the proper coloring of the finite graphs $\Gamma_{1,n}$ and $\Gamma^{un}_{2,n}$ for $n \geq 1.$
\begin{theorem}\label{chromatic thm}
(a) For any $d,$ $\chi (\Gamma_{1,n})=2$ for all $n \geq 1.$
\vskip 0.1in
(b) For any $d,$ $\chi (\Gamma^{un}_{2,n}) = 2n+1$ for any $n \geq 1.$
\end{theorem}
\begin{proof}
(a) For any $d,$ since $\Gamma_{1,n} = K_{2n,2n},$ we have $\chi(\Gamma_{1,n})=\chi(K_{2n,2n})=2$.
\vskip 0.1in
(b) We consider the following two cases.
\vskip 0.1in
(i) If $d \ne 1, 3$ (resp.\ $d = 3$), then since $\Gamma^{un}_{2,n}=K_{2n}*(2n K_1)$ (resp.\ $\Gamma^{un}_{2,n}=(K_{2n}*(2n K_1)) \cup (K_{2n}*(2n K_1)) \cup (K_{2n}*(2n K_1))$), it suffices to compute $\chi(K_{2n}*(2n K_1)).$ Clearly, we need at least $2n$ different colors to color $ K_{2n}$ properly, and then, we only need one additional different color to properly color the join $K_{2n} * (2nK_1),$ and hence, we obtain that $\chi (K_{2n}* (2n K_1)) \geq 2n+1.$ Then from the fact that $\chi(K_{2n})=2n,$ it follows that $\chi(\Gamma^{un}_{2,n})=2n+1.$  
\vskip 0.1in
(ii) If $d = 1$, then we recall that $\Gamma^{un}_{2,n}= 2nK_1 * (K_{2n} \cup 4n K_1)$. Clearly, we need at least $2n$ different colors to properly color $ K_{2n},$ and then, we only need one additional different color to color the remaining $2n K_1$ and $4n K_1$ properly, and hence, we obtain that $\chi ( 2n K_1 * (K_{2n} \cup 4n K_1 ) \geq 2n+1.$ Then again from the fact that $\chi(K_{2n})=2n,$ it follows that $\chi(\Gamma^{un}_{2,n})=2n+1.$
\vskip 0.1in
This completes the proof.
\end{proof}


Next, we compute the clique number of those two graphs $\Gamma_{1,n}$ and $\Gamma^{un}_{2,n}$ for $n \geq 1.$
\begin{theorem}\label{clique thm}
(a) For any $d,$ $\omega (\Gamma_{1,n}) = 2$ for all $n \geq 1.$
\vskip 0.1in
(b) For any $d,$ $\omega(\Gamma^{un}_{2,n})=2n+1$ for all $n \geq 1.$
\end{theorem}
\begin{proof}
(a) Since it is well known that $\chi(G) \geq \omega(G)$ for any graph $G$, the desired assertion follows from Theorem \ref{chromatic thm}-(a), together with the observation that $\Gamma_{1,n}$ contains a $K_2$ for any $n \geq 1$.
\vskip 0.1in
(b) Again, as in part (a), we know that $\omega(\Gamma^{un}_{2,n}) \leq \chi(\Gamma^{un}_{2,n})=2n+1$ for any $d$ and $n \geq 1.$ Now, for any $d$ and $n \geq 1,$ $K_{2n}$ is a subgraph of $\Gamma^{un}_{2,n}$, and hence, $\omega(\Gamma^{un}_{2,n}) \geq 2n.$ But then by the definition of the join of two graphs, we can add at least (in fact, exactly) one more vertex from $2n K_1$ to $K_{2n}$ to obtain a $K_{2n+1}$ so that $\omega(\Gamma^{un}_{2,n}) \geq 2n+1,$ which in turn, implies that $\omega(\Gamma^{un}_{2,n})=2n+1,$ as desired.
\vskip 0.1in
This completes the proof.
\end{proof}

\begin{corollary}
For any $d,$ the finite graphs $\Gamma_{1,n}$ and $\Gamma^{un}_{2,n}$ are both perfect graphs.
\end{corollary}
\begin{proof}
The fact that $\Gamma_{1,n}$ is a perfect graph follows from the observations that $\Gamma_{1,n}=K_{2n,2n}$ and a complete bipartite graph is perfect. Now, we deal with the graph $\Gamma^{un}_{2,n}.$ We consider the following three cases.
\vskip 0.1in
(i) For $d \ne 1,3,$ we know that $\Gamma^{un}_{2,n} = (2n K_1)*K_{2n}.$ Now, suppose on the contrary that $\Gamma^{un}_{2,n}$ is not perfect. Then by definition, there exists an induced subgraph $\Gamma$ of $\Gamma^{un}_{2,n}$ such that $\chi (\Gamma) > \omega (\Gamma).$ Now, if $V(\Gamma)$ is obtained by removing $m_1$ vertices from $V(2n K_1)$ and $m_2$ vertices from $V(K_{2n}),$ then we see that $\Gamma = (2n-m_1)K_1 * K_{2n-m_2}.$ Then as in the proof of Theorems \ref{chromatic thm} and \ref{clique thm}, we obtain
$$\chi(\Gamma) = \begin{cases} 2n-m_2 & \mbox{if}~2n-m_1 =0 \\ 2n-m_2 +1 & \mbox{if}~2n-m_1 > 0 \end{cases}$$
and
$$\omega(\Gamma) = \begin{cases} 2n-m_2 & \mbox{if}~2n-m_1 =0 \\ 2n-m_2 +1 & \mbox{if}~2n-m_1 > 0 \end{cases},$$
which contradicts our assumption that $\chi(\Gamma) > \omega (\Gamma).$
\vskip 0.1in
(ii) For $d=1,$ we have $\Gamma^{un}_{2,n} = 2n K_1 * (K_{2n} \cup 4n K_1).$ Now, suppose on the contrary that $\Gamma^{un}_{2,n}$ is not perfect. Then by definition, there exists an induced subgraph $\Gamma$ of $\Gamma^{un}_{2,n}$ such that $\chi (\Gamma) > \omega (\Gamma).$ Now, if $V(\Gamma)$ is obtained by removing $m_1$ vertices from $V(2n K_1)$, $m_2$ vertices from $V(K_{2n}),$ and $m_3$ vertices from the remaining $4nK_1$, then we see that $\Gamma = (2n-m_1)K_1 * (K_{2n-m_2} \cup (4n-m_3)K_1).$ Then as in the proof of Theorems \ref{chromatic thm} and \ref{clique thm}, we obtain
$$\chi(\Gamma) = \begin{cases} 0 & \mbox{if}~2n-m_1 = 2n-m_2 = 4n-m_3 = 0 \\ 1 & \mbox{if}~2n-m_1 = 2n-m_2 =0, 4n-m_3 > 0 \\  2n-m_2 & \mbox{if}~2n-m_1 = 0, 2n-m_2 > 0 \\ 1 & \mbox{if}~2n-m_1 > 0, 2n-m_2 = 4n-m_3 = 0  \\ 2 & \mbox{if}~2n-m_1 > 0, 2n-m_2 = 0, 4n-m_3 > 0  \\ 2n-m_2+1 & \mbox{if}~2n-m_1 > 0, 2n-m_2 >0 \end{cases}$$
and
$$\omega(\Gamma) = \begin{cases} 0 & \mbox{if}~2n-m_1 = 2n-m_2 = 4n-m_3 = 0 \\ 1 & \mbox{if}~2n-m_1 = 2n-m_2 =0, 4n-m_3 > 0 \\  2n-m_2 & \mbox{if}~2n-m_1 = 0, 2n-m_2 > 0 \\ 1 & \mbox{if}~2n-m_1 > 0, 2n-m_2 = 4n-m_3 = 0  \\ 2 & \mbox{if}~2n-m_1 > 0, 2n-m_2 = 0, 4n-m_3 > 0  \\ 2n-m_2+1 & \mbox{if}~2n-m_1 > 0, 2n-m_2 >0 \end{cases},$$
which contradicts our assumption that $\chi(\Gamma) > \omega (\Gamma).$
\vskip 0.1in
(iii) For $d=3,$ we recall that $\Gamma^{un}_{2,n}=((2n K_1)* K_{2n}) \cup ((2n K_1)* K_{2n}) \cup ((2n K_1)* K_{2n})$. For ease of notation, let $\Gamma^j$ be the $j$-th $(2n K_1)*K_{2n}$ in $\Gamma^{un}_{2,n}$ for $j=1,2,3.$ Then note that, for any induced subgraph $\Gamma$ of $\Gamma^{un}_{2,n},$ $\Gamma$ is a disjoint union of $\Gamma \cap \Gamma^1$, $\Gamma \cap \Gamma^2$, and $\Gamma \cap \Gamma^3$ (i.e.\ we have $\Gamma = (\Gamma \cap \Gamma^1 ) \cup (\Gamma \cap \Gamma^2) \cup (\Gamma \cap \Gamma^3)$), and hence, it follows that 
$$\chi (\Gamma)= \max(\chi(\Gamma \cap \Gamma^1), \chi(\Gamma \cap \Gamma^2), \chi(\Gamma \cap \Gamma^3)) = \max(\omega(\Gamma \cap \Gamma^1), \omega(\Gamma \cap \Gamma^2), \omega(\Gamma \cap \Gamma^3)) = \omega (\Gamma),$$
where the equality in the middle follows from a similar argument as in the case (i) above, showing that $\chi (\Gamma \cap \Gamma^j)=\omega (\Gamma \cap \Gamma^j)$ for each $j=1,2,3$.
\vskip 0.1in
This completes the proof.
\end{proof}
The next result is related to a maximum independent set of the graphs $\Gamma_{1,n}$ and $\Gamma_{2,n}^{un}$ for each $n \geq 1.$
\begin{theorem}\label{independence thm}
(a) For any $d,$ $\alpha (\Gamma_{1,n})=2n$ for all $n \geq 1.$
\vskip 0.1in
(b) For $d \ne 1, 3,$ $\alpha (\Gamma^{un}_{2,n}) = 2n$ for any $n \geq 1.$
\vskip 0.1in
(c) For $d = 1,$ $\alpha (\Gamma^{un}_{2,n}) = 4n+1$ for any $n \geq 1.$
\vskip 0.1in
(d) For $d = 3,$ $\alpha (\Gamma^{un}_{2,n}) = 6n$ for any $n \geq 1.$
\end{theorem}
\begin{proof}
(a) For any $d,$ since $\Gamma_{1,n} = K_{2n,2n},$ we have $\alpha(\Gamma_{1,n})=\alpha(K_{2n,2n})=2n$.
\vskip 0.1in
(b) If $d \ne 1, 3,$ then since $\Gamma^{un}_{2,n}=K_{2n}*(2n K_1)$, it suffices to compute $\alpha(K_{2n}*(2n K_1)).$ Let $S$ be a maximum independent set of $K_{2n} * (2n K_1).$ We observe that if $S$ contains any vertex from $K_{2n},$ then we have $|S|=1.$ In particular, by the definition of a maximum independent set, $S$ must not contain any vertex from $K_{2n},$ and hence, we have $|S| \leq 2n.$ On the other hand, since any two distinct vertices in $2n K_1$ are not adjacent to each other, we also have $|S| \geq 2n.$ It follows that we have $|S|=2n.$
\vskip 0.1in
(c) If $d = 1$, then we recall that $\Gamma^{un}_{2,n}= 2nK_1 * (K_{2n} \cup 4n K_1)$. Again, let $S$ be a maximum independent set of $2n K_1 * (K_{2n} \cup 4n K_1).$ As in the proof of part (b), we can see that $|S| \geq \max (2n, 4n+1)=4n+1,$ where $4n+1$ comes from the cardinality of a maximum independent set $T$ of $K_{2n} \cup 4n K_1,$ with the property that $S$ contains $T.$ Now, if $S \ne T,$ then we must add at least one vertex from $2n K_1$ to $T$ to get $S$, which is impossible because those added vertices from $2n K_1$ are adjacent to each vertex in $T.$ It follows that $S=T,$ and hence, we get $|S|=4n+1.$
\vskip 0.1in
(d) If $d = 3,$ then since $\Gamma^{un}_{2,n}=(K_{2n}*(2n K_1)) \cup (K_{2n}*(2n K_1)) \cup (K_{2n}*(2n K_1))$, the desired result essentially follows from part (b).
\vskip 0.1in
This completes the proof.
\end{proof}

Finally, we also compute the vertex connectivity of $\Gamma_{1,n}$ and $\Gamma^{un}_{2,n}$ for $n \geq 1.$
\begin{theorem}\label{connectivity thm}
(a) For any $d,$ $\kappa(\Gamma_{1,n})=2n$ for all $n \geq 1.$
\vskip 0.1in
(b) For $d \ne 3,$ $\kappa(\Gamma^{un}_{2,n})=2n$ for all $n \geq 1.$
\vskip 0.1in
(c) For $d=3,$ $\kappa(\Gamma_j)=2n$ for all $n \geq 1$ and for each $j=2,3,4.$ (For the definition of $\Gamma_j$, we refer to (Case 3) of Theorem \ref{main thm1}.)
\end{theorem}
\begin{proof}
(a) For any $d,$ we know that $\Gamma_{1,n}=K_{2n, 2n}.$ For notational convenience, let $V_1, V_2$ be the two subsets of $V(K_{2n,2n})$ such that every pair of vertices either in $V_1$ or $V_2$ is not adjacent to each other. If we remove all the vertices in $V_1$, then clearly, the remaining graph is disconnected. Hence, we have $\kappa(K_{2n,2n}) \leq 2n$, by definition. Now, suppose that $\kappa(K_{2n,2n}) < 2n.$ Then there is a vertex cut obtained by removing $m_1$ vertices from $V_1$ and $m_2$ vertices from $V_2$ with $m_1 + m_2 < 2n,$ that results in a disconnected graph. But then we note that the remaining graph becomes the complete bipartite graph $K_{2n-m_1, 2n-m_2}$ so that it is connected, which gives a contradiction. Hence, we can see that $\kappa(K_{2n,2n})=2n.$
\vskip 0.1in
(b) We consider the following two cases.
\vskip 0.1in
(i) For $d \ne 1,3,$ we know that $\Gamma^{un}_{2,n} = (2n K_1)*K_{2n}.$ For notational convenience, let $V_1=V(2nK_1)$ and $V_2=V(K_{2n})$ be the two subsets of $V(2nK_1 * K_{2n})$. If we remove all the vertices in $V_2$, then clearly, the remaining graph is disconnected. Hence, we have $\kappa((2nK_1) * K_{2n}) \leq 2n$, by definition. Now, suppose that $\kappa((2nK_1)*K_{2n}) < 2n.$ Then there is a vertex cut obtained by removing $m_1$ vertices from $V_1$ and $m_2$ vertices from $V_2$ with $m_1 + m_2 < 2n,$ that results in a disconnected graph. But then we note that the remaining graph becomes the graph $(2n-m_1)K_1 * K_{2n-m_2}$ so that it is connected, which gives a contradiction. Hence, we can see that $\kappa((2nK_1) * K_{2n})=2n.$

\vskip 0.1in
(ii) For $d=1,$ we have $\Gamma^{un}_{2,n} = 2n K_1 * (K_{2n} \cup 4n K_1).$ For notational convenience, let $V_1=V(2nK_1)$, $V_2=V(K_{2n})$, and $V_3 = V(4n K_1)$ be the three subsets of $V(2nK_1 * (K_{2n} \cup 4n K_1))$. If we remove all the vertices in $V_1$, then clearly, the remaining graph is disconnected. Hence, we have $\kappa(2n K_1 * (K_{2n} \cup 4n K_1)) \leq 2n$, by definition. Now, suppose that $\kappa(2n K_1 * (K_{2n} \cup 4n K_1)) < 2n.$ Then there is a vertex cut obtained by removing $m_1$ vertices from $V_1$, $m_2$ vertices from $V_2$, and $m_3$ vertices from $V_3$ with $m_1 + m_2 +m_3 < 2n,$ that results in a disconnected graph. But then we note that the remaining graph becomes the graph $(2n-m_1)K_1 * (K_{2n-m_2} \cup (4n-m_3) K_1)$ so that it is connected, which gives a contradiction. Hence, we can see that $\kappa(2nK_1 *( K_{2n} \cup 4n K_1) )=2n.$
\vskip 0.1in
(c) For $d=3,$ and for each $j=2,3,4,$ we know that $\Gamma_j = (2nK_1)*K_{2n},$ and hence, we get $\kappa(\Gamma_j)=2n$ for each $j=2,3,4$ by the same argument as in the proof of part (b)-(i) above.
\vskip 0.1in
This completes the proof.
\end{proof}
In particular, both of the graphs $\Gamma_{1,n}$ and $\Gamma^{un}_{2,n}$ (for $d=3,$ $\Gamma_2, \Gamma_3,$ and $\Gamma_4$) are $2$-connected for any $n \geq 1.$

\section{Automorphism groups and the Jordan property}\label{aut}
In this section, we recall some facts from the theory of Jordan groups, following \cite{5}, and determine whether the automorphism group of the zero-divisor graph of $M_2(R)$ for some $R=\mathcal{O}_K$ with $K=\Q(\sqrt{-d})$ either for $d=0$ or $d>0$ a square-free integer is a Jordan group or not. First, we recall the definitions of Jordan groups and Jordan constants.
 
\begin{definition}[{\cite[Definition 1]{5}}]\label{Jordan}
A group $G$ is called a \emph{Jordan group} if there exists an integer $d>0$, depending only on $G$, such that every finite subgroup $H$ of $G$ contains a normal abelian subgroup whose index in $H$ is at most $d.$ The minimal such an integer $d$ is called the \emph{Jordan constant of $G$} and is denoted by $J_{G}$.
\end{definition}
In particular, if we know the structure of finite subgroups of a given group $G$, then we can talk about the Jordaness of $G$, and potentially, we can further compute its Jordan constant $J_G$ in case $G$ is a Jordan group. For example, it is rather easy to see that if every finite subgroup of $G$ is abelian, then $G$ is a Jordan group, and in fact, we have $J_G =1.$
\vskip 0.1in
We summarize one useful observation in the following remark.
\begin{remark}\label{simple rmk}
Let $G$ be a finite group. Then by definition, $G$ is a Jordan group with $J_G \leq |G|.$ Moreover we have:
\vskip 0.1in
(1) $G$ is abelian if and only if we have $J_G=1.$ 
\vskip 0.1in
(2) If $G$ is simple i.e.\ the only normal subgroups of $G$ are the trivial group and $G$ itself, then we have $J_G = |G|.$

\end{remark}
The following example, which makes use of the simpleness of the alternating group $A_n$ for $n \geq 5$, will be used in the proof of Theorem \ref{main thm2} below.
\begin{example}\label{alt exam}
For $n \geq 5,$ we have $J_{A_n}=\frac{n!}{2}.$ Indeed, we know that $A_n$ is simple for any $n \geq 5$, and hence, we get $$J_{A_n}=  |A_n | = \frac{n!}{2}$$
by Remark \ref{simple rmk}-(2).
\end{example}
For our later use, we also record one basic property of the Jordaness of groups.
\begin{lemma}\cite[Theorem 3]{5}\label{sub jordan lem}
Let $G$ be a Jordan group, and let $H$ be a subgroup of $G.$ Then $H$ is also a Jordan group, and we have
$$J_H \leq J_G.$$
\end{lemma}

Now, we first compute the automorphism group of the induced subgraphs $\Gamma_{1,n}$ and $\Gamma^{un}_{2,n}$ of $\Gamma(M_2(R))$ for each $n \geq 1.$
\begin{theorem}\label{aut thm}
Let $G_{1,n}=\textrm{Aut}(\Gamma_{1,n})$ and $G_{2,n}=\textrm{Aut}(\Gamma^{un}_{2,n})$ for each $n \geq 1.$
\vskip 0.1in
(a) For any $d,$ $G_{1,n} \cong (S_{2n} \times S_{2n}) \rtimes C_2.$
\vskip 0.1in
(b) For $d \ne 1,3,$ $G_{2,n} \cong S_{2n} \times S_{2n}.$
\vskip 0.1in
(c) For $d=1,$ $G_{2,n} \cong S_{2n} \times S_{2n} \times S_{4n}.$
\vskip 0.1in
(d) For $d=3,$ $G_{2,n} \cong (S_{2n}^2 \times S_{2n}^2 \times S_{2n}^2) \rtimes S_3,$ where $S_{2n}^2 = S_{2n} \times S_{2n}.$ 
\end{theorem}
\begin{proof}
(a) This follows from the fact that $\Gamma_{1,n} = K_{2n, 2n}$ for any $d$ and $n \geq 1,$ and the well-known fact that $\textrm{Aut}(K_{2n,2n}) \cong (S_{2n} \times S_{2n}) \rtimes C_2.$
\vskip 0.1in
(b) For $d \ne 1,3,$ we know that $\Gamma^{un}_{2,n} = (2n K_1)*K_{2n}$ for each $n \geq 1.$ Note that $\textrm{Aut}(2n K_1)=\textrm{Aut}(K_{2n}) \cong S_{2n}.$ Now, for convenience, we let $V(2n K_1) = \{v_1, v_2, \cdots, v_{2n} \}$ and $V(K_{2n})=\{w_1, w_2, \cdots, w_{2n}\}.$ Then for any $\sigma, \tau \in S_{2n},$ we define a bijection $(\sigma, \tau) \colon V(2n K_1 * K_{2n}) \rightarrow V(2n K_1 * K_{2n})$ by
$$(\sigma, \tau)(x)=\begin{cases} v_{\sigma(j)} & \mbox{if $x=v_j$ for some $j \in \{1, 2, \cdots, 2n\}$} \\ w_{\tau(j)} & \mbox{if $x=w_j$ for some $j \in \{1, 2, \cdots, 2n\}$} \end{cases}.$$
We claim that $(\sigma, \tau)$ is an automorphism of the graph $(2n K_1)* K_{2n}.$ Indeed, let $u, v \in V(2n K_1 * K_{2n})$, and we consider the following three cases:
\vskip 0.1in
(i) If $u, v  \in V(2n K_1),$ so that $u=v_j$ and $v=v_k$ for some $j, k \in \{1, 2, \cdots, 2n\}$, then we have $(\sigma, \tau)(u)=v_{\sigma(j)}$ and $(\sigma, \tau)(v)=v_{\sigma(k)}$ so that $u$ and $v$ (resp.\ $(\sigma, \tau)(u)$ and $(\sigma, \tau)(v)$) are not adjacent to each other at the same time.
\vskip 0.1in
(ii) If exactly one of $u$ and $v$ is in $V(2n K_1),$ then without loss of generality, assume further that $u  \in V(2nK_1)$ and $v \in V(K_{2n}).$ Then we may write $u=v_j$ and $v=w_k$ for some $j, k \in \{1,2,\cdots, 2n\}.$ It follows then that we have $(\sigma, \tau)(u)=v_{\sigma(j)}$ and $(\sigma, \tau)(v)=w_{\tau(k)}$, and hence, we see that $u$ and $v$ (resp.\ $(\sigma, \tau)(u)$ and $(\sigma, \tau)(v)$) are adjacent to each other at the same time.
\vskip 0.1in
 (iii) If $u, v \in V(K_{2n}),$ so that $u=w_j$ and $v=w_k$ for some $j, k \in \{1,2,\cdots, 2n\},$ then we have $(\sigma, \tau)(u)=w_{\tau(j)}$ and $(\sigma, \tau)(v)=w_{\tau(k)}$ so that $u$ and $v$ (resp.\ $(\sigma, \tau)(u)$ and $(\sigma, \tau)(v)$) are adjacent to each other at the same time due to the fact that $K_{2n}$ is complete.

\vskip 0.1in
Hence, it follows from (i), (ii), and (iii) that for any $u, v \in V(2nK_1 * K_{2n}),$ we have that $u$ is adjacent to $v$ if and only if $(\sigma, \tau)(u)$ is adjacent to $(\sigma, \tau)(v),$ and hence, we conclude that $(\sigma, \tau)$ is an automorphism of $(2nK_1)* K_{2n},$ and this construction shows that the map
$$\varphi \colon S_{2n} \times S_{2n} \rightarrow \textrm{Aut}((2nK_1) * K_{2n})$$
given by $\varphi((\sigma, \tau))= (\sigma, \tau)$ (where the latter denotes the automorphism of $(2nK_1)*K_{2n}$ described above) is well-defined and injective. Now, it remains to show that $\varphi$ is surjective. To this aim, let $f \in \textrm{Aut}(2nK_1 * K_{2n}).$ We first show that $f(2nK_1) = 2nK_1$ and $f(K_{2n})=K_{2n}.$ Indeed, if $f(v_j)=w_k$ for some $j, k \in \{1,2,\cdots, 2n\},$ then it follows that $f(v_l)$ is adjacent to $f(v_j)=w_k$ for any $l \ne j$ (by definition), and then, since $f$ is an automorphism of $(2nK_1) * K_{2n}$, it implies that $v_l$ is adjacent to $v_j,$ which gives a contradiction. Thus $f(2nK_1)=2nK_1$ and $f(K_{2n})=K_{2n},$ as desired. Now, in view of the above observation, consider the natural restriction maps $f|_{2n K_1}$ and $f|_{K_{2n}}$ of $f$. Then by definition, it is clear that $f|_{2nK_1} \in \textrm{Aut}(2nK_1) \cong S_{2n}$ and $f|_{K_{2n}} \in \textrm{Aut}(K_{2n}) \cong S_{2n}.$ By construction, we have $\varphi(f|_{2nK_1}, f|_{K_{2n}})=f,$ and this proves the surjectivity of $\varphi,$ as desired. Hence, we conclude that $\textrm{Aut}(2nK_1 * K_{2n}) \cong S_{2n} \times S_{2n}.$
\vskip 0.1in
(c) For $d=1,$ we recall that $\Gamma^{un}_{2,n} = 2n K_1 * (K_{2n} \cup 4n K_1)$ for each $n \geq 1.$ Note that $\textrm{Aut}(2n K_1)=\textrm{Aut}(K_{2n}) \cong S_{2n}$ and $\textrm{Aut}(4nK_1) \cong S_{4n}.$ Now, for convenience, we let $V(2n K_1) = \{u_1, u_2, \cdots, u_{2n} \}$ and $V(K_{2n})=\{v_1, v_2, \cdots, v_{2n}\}$, and $V(4nK_1)=\{w_1, w_2, \cdots, w_{4n}\}.$ Then for any $\sigma, \tau \in S_{2n}$ and $\eta \in S_{4n},$ we define a bijection $(\sigma, \tau, \eta) \colon V(2n K_1 *( K_{2n} \cup 4n K_1)) \rightarrow V(2n K_1 *( K_{2n} \cup 4n K_1))$ by
$$(\sigma, \tau, \eta)(x)=\begin{cases} u_{\sigma(j)} & \mbox{if $x=u_j$ for some $j \in \{1, 2, \cdots, 2n\}$} \\ v_{\tau(j)} & \mbox{if $x=v_j$ for some $j \in \{1, 2, \cdots, 2n\}$} \\ w_{\eta(j)} & \mbox{if $x=w_j$ for some $j \in \{1,2,\cdots, 4n\}$} \end{cases}.$$
Then we can proceed as in the proof of part (b) above to show that $$\textrm{Aut}(2nK_1 * (K_{2n} \cup 4n K_1)) \cong S_{2n} \times S_{2n} \times S_{4n}.$$

(d) For $d=3,$ we recall that $\Gamma^{un}_{2,n}=(2n K_1)* K_{2n} \cup (2n K_1)* K_{2n} \cup (2n K_1)* K_{2n}.$ For ease of notation, as before, let $\Gamma^j$ be the $j$-th $(2n K_1)*K_{2n}$ in $\Gamma^{un}_{2,n}$ for $j=1,2,3.$ Then by first permuting $\Gamma^1, \Gamma^2,$ and $\Gamma^3,$ and then,  by applying the automorphisms within each $\Gamma^1, \Gamma^2,$ and $\Gamma^3$ (using a similar argument as in the proof of part (b)), we obtain
$$\textrm{Aut}(\Gamma^{un}_{2,n}) \cong (S_{2n}^2 \times S_{2n}^2 \times S_{2n}^2) \rtimes S_3,$$
as desired.
\vskip 0.1in
This completes the proof.
\end{proof}

We conclude this paper by proving the non-Jordaness of the automorphism group of the zero-divisor graph of $M_2(R)$. 
\begin{theorem}\label{main thm2}
Let $G= \textrm{Aut}(\Gamma(M_2 (R))).$ Then $G$ is not a Jordan group.
\end{theorem}
\begin{proof}
We first claim that $S_{2n}$ is (isomorphic to) a subgroup of $G$ for each $n \geq 1.$ Indeed, consider the induced subgraph $\Gamma^{\prime}=K_{2n}$ of $\Gamma^{un}_{2,n}$ whose vertices lie on $S_{\textrm{tc},1,-1, n}$ in Definition \ref{level def}. Then we can construct automorphisms $\sigma$ of $\Gamma(M_2(R))$ as follows: let $\sigma \colon V(\Gamma(M_2(R))) \rightarrow V(\Gamma(M_2(R)))$ be a bijective map that permutes the $2n$ vertices in $\Gamma^{\prime}$ and fixes all the other vertices in $\Gamma(M_2(R)).$ Now, we need to show that $\sigma$ is an automorphism of the graph $\Gamma(M_2(R)).$ To this aim, let $v_1, v_2 \in V(\Gamma(M_2(R)))$, and we consider the following three cases:
\vskip 0.1in
(i) If $v_1, v_2 \not \in \Gamma^{\prime},$ then we have $\sigma(v_1)=v_1$ and $\sigma(v_2)=v_2$ so that $v_1$ is adjacent to $v_2$ if and only if $\sigma(v_1)$ is adjacent to $\sigma(v_2)$.
\vskip 0.1in
(ii) If exactly one of $v_1$ and $v_2$ is in $\Gamma^{\prime},$ then without loss of generality, assume further that $v_1 \not \in \Gamma^{\prime}$ and $v_2 \in \Gamma^{\prime}.$ Then we have $\sigma(v_1)=v_1$ and $\sigma(v_2) \in \Gamma^{\prime}.$ Then by the description of $\Gamma^{\prime}$ (see Definition \ref{level def} above), we may write $v_2 =  \lambda \begin{bmatrix} 1 & -1 \\ 1 & - 1 \end{bmatrix}$ and $\sigma(v_2) = \mu \begin{bmatrix}1 & -1 \\ 1 & - 1 \end{bmatrix}$ for some $\lambda, \mu \in \Z \setminus \{0\}$. Now, note that
$$ v_1 v_2 = \lambda v_1 \cdot \begin{bmatrix} 1 & -1 \\ 1 & -1 \end{bmatrix} = O~~\textrm{if and only if}~~\sigma(v_1)\sigma(v_2) = v_1 \sigma(v_2) = \mu v_1 \cdot \begin{bmatrix} 1 & -1 \\ 1 & -1 \end{bmatrix} = O.$$
This means that $v_1$ and $v_2$ are adjacent if and only if $\sigma(v_1)$ and $\sigma(v_2)$ are adjacent.
\vskip 0.1in
 (iii) If $v_1, v_2 \in \Gamma^{\prime},$ then $\sigma(v_1), \sigma(v_2) \in \Gamma^{\prime}.$ Then since $\Gamma^{\prime} = K_{2n}$ is complete, we can see that $v_1$ and $v_2$ are adjacent if and only if $\sigma(v_1)$ and $\sigma(v_2)$ are adjacent.
\vskip 0.1in
Hence, it follows from (i), (ii), and (iii) that $\sigma$ is an automorphism of $\Gamma(M_2(R)),$ and this construction shows that $S_{2n}$ is (isomorphic to) a subgroup of $G$, as desired.  
\vskip 0.1in

Now, suppose on the contrary that $G$ is a Jordan group. Then for any $d$, since $S_{2n} \leq G$ for any $n \geq 1,$ we further have that $A_{2n} \leq G$ for any $n \geq 1.$ Then it follows from Example \ref{alt exam} and Lemma \ref{sub jordan lem} that we have
$$ J_{G} \geq J_{A_{2n}} = \frac{(2n)!}{2} \rightarrow \infty~~\textrm{as}~n \rightarrow \infty,$$
(where we implicitly use the fact that $A_{2n}$ is simple for any $n \geq 3$), which is absurd, and hence, we conclude that $G$ cannot be a Jordan group. 
\vskip 0.1in
This completes the proof.
\end{proof}

\begin{remark}
There are some quite large infinite groups which turn out to be Jordan groups. Hence, it is not at all trivial to see that $\textrm{Aut}(\Gamma(M_2(R)))$ is not a Jordan group, a priori.
\end{remark}
Motivated by Theorem \ref{main thm2}, we conclude this paper with the following question.

\begin{problem}\label{aut prob}
For which commutative rings $R$ with identity, does there exist an integer $n_0 \geq 1$ such that $\textrm{Aut}(\Gamma(M_2(R)))$ contains an $S_n$ (as a subgroup) for all $n \geq n_0?$
\end{problem}
Note that if such an $n_0$ exists for a commutative ring $R$ with identity, then $\textrm{Aut}(\Gamma(M_2(R)))$ is not Jordan.
\begin{remark}
If $R$ is a finite commutative ring with identity, then the vertex set $V(\Gamma(M_2(R)))$ is also a finite set, say, with $|V(\Gamma(M_2(R)))| = m.$ Then it follows that $\textrm{Aut}(\Gamma(M_2(R)))$ cannot contain an $S_n$ (as a subgroup) for all $ n \geq m+1,$ and hence, there is no such an integer $n_0$ as in Problem \ref{aut prob}.
\end{remark}


\section{Conclusion}
In this paper, we constructed several induced subgraphs of the zero-divisor graph $\Gamma(M_2(R))$ of $M_2(R)$ for some ring $R$ of algebraic integers of a number field $K=\Q(\sqrt{-d})$ for either $d=0$ or $d>0$ a square-free integer. We examined some of the properties, and computed the automorphism group of those induced subgraphs of $\Gamma(M_2(R))$, and we further show that the automorphism group of $\Gamma(M_2(R))$ is not a Jordan group. These results are summarized in the following remark. 
\begin{remark}
Let $R$ be the ring of integers of a number field $K=\Q(\sqrt{-d})$ for either $d=0$ or $d>0$ a square-free integer. Let $\Gamma_{1,n}$ and $\Gamma_{2,n}$ be the induced subgraphs of $\Gamma(M_2(R))$ given in Definition \ref{subgraph def}. Then we have the following table.
  \begin{center}
  \begin{tabular}{c | c | c | c }
\hline
$$ & $d \ne 1, 3$ & $d=1$ &$d=3$  \\
\hline
$g(\Gamma_{1,n})$ & $4$ & $4$ &$4$ \\
\hline
$g(\Gamma^{un}_{2,n})$ & $3$ & $3$ &$3$ \\
\hline
$\chi(\Gamma_{1,n})$ & $2$ & $2$ &$2$ \\
\hline
$\chi(\Gamma^{un}_{2,n})$ & $2n+1$  & $2n+1$ &$2n+1$ \\
\hline
$\omega(\Gamma_{1,n})$ & $2$ & $2$ &$2$\\
\hline
$\omega(\Gamma^{un}_{2,n})$ & $2n+1$ & $2n+1$ &$2n+1$ \\
\hline
$\alpha(\Gamma_{1,n})$ & $2n$ & $2n$ &$2n$\\
\hline
$\alpha(\Gamma^{un}_{2,n})$ & $2n$ & $4n+1$ &$6n$ \\
\hline
$\kappa(\Gamma_{1,n})$ & $2n$ & $2n$ &$2n$\\
\hline
$\kappa(\Gamma^{un}_{2,n})$ & $2n$ & $2n$ &$2n$\\
\hline
$\textrm{Aut}(\Gamma_{1,n})$ & $(S_{2n} \times S_{2n}) \rtimes C_2$ & $(S_{2n} \times S_{2n}) \rtimes C_2$ &$(S_{2n} \times S_{2n}) \rtimes C_2$\\
\hline
$\textrm{Aut}(\Gamma^{un}_{2,n})$ & $S_{2n} \times S_{2n}$ & $S_{2n} \times S_{2n} \times S_{4n}$  &$(S_{2n}^2 \times S_{2n}^2 \times S_{2n}^2) \rtimes S_3 $\\
\hline
Is $\Gamma_{1,n}$ a cograph? & Yes & Yes & Yes \\
\hline
Is $\Gamma^{un}_{2,n}$ a cograph? & Yes & Yes & Yes \\
\hline
Is $\textrm{Aut}(\Gamma(M_2(R)))$ Jordan? & No & No & No \\
\hline
\end{tabular}
\vskip 4pt
\textnormal{Table 1} \\
\textnormal{Properties of the induced subgraphs $\Gamma_{1,n}$ and $\Gamma^{un}_{2,n}$, and their automorphism groups.}
\end{center}
\end{remark}

\textbf{Acknowledgments}
\vskip 0.1in
The authors are grateful to Dr. Minho Cho and Professor Hyungkee Yoo for their valuable comments on the paper.

\Addresses


\begin{thebibliography}{00}
\bibitem{12}
{S. Akbari, H. R. Maimani, and S. Yassemi,} When a zero-divisor graph is planar or a complete $r$-partite graph, J. Algebra \textbf{270} (2003), no.\ 1, 169--180.
\bibitem{10}
{S. Akbari and A. Mohammadian,} On the zero-divisor graph of a commutative ring, J. Algebra \textbf{274} (2004), no.\ 2, 847--855.
\bibitem{14}
{S. Akbari and A. Mohammadian,} Zero-divisor graphs of non-commutative rings, J. Algebra \textbf{296} (2006), no.\ 2, 462--479.
\bibitem{8}
{D. D. Anderson and M. Naseer,} Beck's coloring of a commutative ring, J. Algebra \textbf{159} (1993), no.\ 2, 500--514.
\bibitem{3}
{D. F. Anderson and P. S. Livingston,} The zero-divisor graph of a commutative ring, J. Algebra \textbf{217} (1999), no.\ 2, 434--447.
\bibitem{11}
{D. F. Anderson, A. Frazier, A. Lauve, and P. S. Livingston,} The zero-divisor graph of a commutative ring II, Proceedings, Ideal Theoretic Methods in Commutative Algebra; (Columbia, MO, 1999) Lecture Notes in Pure and Appl. Math. \textbf{220} (2001), 61--72.
\bibitem{7}
{I. Beck,} Coloring of commutative rings, J. Algebra \textbf{116} (1988), no.\ 1, 208--226.
\bibitem{6}
{I. Bozic and Z. Petrovic,} Zero-divisor graphs of matrices over commutative rings, Comm. Algebra, \textbf{37} (2009), no.\ 4, 1186--1192. 
\bibitem{13}
{F. DeMeyer and K. Schneider,} Automorphisms and zero divisor graphs of commutative rings, Commutative rings, Nova Science Publishers, Inc., Hauppauge, NY (2002), 25--37.
\bibitem{1}
{J. Han,} The zero-divisor graph under group actions in a noncommutative ring, J. Korean Math. Soc., \textbf{45} (2008), no.\ 6, 1647--1659.
\bibitem{2}
{J. Harris,} Algebraic geometry. A first course, Grad. Texts in Math., \textbf{133}, Springer-Verlag, New York. (1992), xx+328 pp.

\bibitem{4}
{K. Koh,} On ``Properties of rings with a finite number of zero-divisors", Math. Ann. \textbf{171} (1967), 79--80.
\bibitem{5}
{V. L. Popov,} Jordan Groups and Automorphism Groups of Algebraic Varieties, Automorphisms in birational and affine geometry, (I. Cheltsov et al., eds.), Springer, Cham, (2014), 185--213.
\bibitem{9}
{S. P. Redmond,} The zero-divisor graph of a non-commutative ring, Internat. J. Commutative Rings, 1 (4) (2002), 203--211.
\bibitem{15}
{L. Wang,} Automorphisms of the zero-divisor graph of the ring of all $n \times n$ matrices over a finite field, Discrete Math, \textbf{339} (2016), no.\ 8, 2036--2041.
\end{thebibliography}
\end{document}